\documentclass{amsart}

\usepackage[utf8]{inputenc}
\usepackage[utf8]{inputenc}   
\usepackage[T1]{fontenc}      
\usepackage{lmodern}

\usepackage{amssymb,amsmath,amsfonts,latexsym}
\usepackage{graphicx}
\usepackage[all]{xy}
\usepackage{kmath}
\usepackage{cite}
\usepackage[T1]{fontenc}
\usepackage{hyperref}
\usepackage[parfill]{parskip}
\usepackage{float}

\usepackage{tikz}
\usepackage{newtxtext,newtxmath}
\usepackage{tkz-euclide}
\usetikzlibrary{decorations.markings,arrows}
\tikzset{
    arrowMe/.style={
        postaction=decorate,
        decoration={
            markings,
            mark=at position .5 with {\arrow[thick]{#1}}
        }
    }
}

\usetikzlibrary{backgrounds,calc}
\tikzset{point/.style = {fill=black,circle,inner sep=0.7pt}}

\pgfarrowsdeclaredouble{<<s}{>>s}{stealth}{stealth}
\pgfarrowsdeclaretriple{<<<s}{>>>s}{stealth}{stealth}

\tikzstyle{vertex}=[circle,fill=black!25,minimum size=20pt,inner sep=0pt]
\tikzstyle{edge} = [draw]
\tikzstyle{weight} = [font=\tiny]

\newcommand{\Aff}{\ensuremath{\text{Aff}}}
\title{The Topology of $3$-Dimensional Hessian Manifolds}

\author{Emmanuel Gnandi}
\address{INSA de Toulouse, Département de Génie Mathématique, IMT, 135 Avenue de Rangueil, 31077 Toulouse Cedex 4, France}
\email{kpanteemmanuel@gmail.com, gnandi@insa-toulouse.fr}

\newtheorem{thm}{Theorem}[section]
\newtheorem{cor}{Corollary}[section]

\newtheorem{prop}{Proposition}[section]
\newtheorem{ex}[thm]{Exemple}
\newtheorem{rem}[thm]{Remark}
\newtheorem{defn}[thm]{Definition}

\newtheorem{com}[thm]{Comments}
\numberwithin{equation}{section}

\begin{document}
\maketitle
\begin{abstract}
We investigate the global topology of $3$-dimensional Hessian manifolds. 
We prove that any compact, orientable $3$-dimensional Hessian manifold is either a 
Hantzsche–\\Wendt manifold or admits the structure of a Kähler mapping torus. 
We analyze the parity of Betti numbers for compact, orientable $3$-dimensional Hessian manifolds, 
with special focus on those of Koszul type (hyperbolic manifolds). 
Moreover, we show that the product of two compact, orientable, $3$-dimensional Hessian manifolds 
of Koszul type naturally carries a Kähler structure. 
Finally, we establish that every compact, orientable, $3$-dimensional Hessian manifold 
is a Seifert manifold with trivial Euler number, 
whose underlying orbifold has either vanishing or negative Euler characteristic, 
thus providing a complete topological classification.
\end{abstract}

\section{Introduction}
Hessian manifolds play a fundamental role in affine differential geometry and information geometry. 
They are defined by a locally flat, torsion-free connection $\nabla$ and a Riemannian metric $g$ such that, locally, there exists a smooth function $\phi$ with
\[
g = \nabla d\phi.
\]
 
In his seminal works~\cite{amari2012differential, amari2016information}, Amari showed that exponential families are always Hessian manifolds, 
since the Fisher information metric can be expressed as the Hessian of the logarithmic partition function. Hessian geometry is deeply connected to Kähler geometry. 
Shima established in~\cite{shima1981hessian, shima1986vanishing, shima1997geometry} that the tangent bundle of a Hessian manifold naturally carries a Kähler structure. 
Moreover, a Hessian structure is formally analogous to a Kähler structure, 
since a Kähler metric is locally given as a complex Hessian.  
Armstrong and Amari~\cite{amari2014curvature} further proved that the Pontryagin forms of any Hessian metric necessarily vanish. 

It is well known that compact manifolds with finite fundamental groups cannot admit Hessian structures~\cite{ay2002dually}. 
Thus, the existence of a Hessian structure imposes strong topological obstructions on the underlying manifold.  
In particular, Shima~\cite{shima1981hessian} proved that the universal affine covering of a compact Hessian manifold is a convex domain, 
and deduced the nonexistence of Hessian metrics on Hopf manifolds 
$(\mathbb{S}^{n-1} \times \mathbb{S}^{1})$ for $n>1$. 

In this paper, we investigate the global structure of compact, orientable $3$-dimensional Hessian manifolds. 
We establish a correspondence, in dimension three, between compact, orientable Hessian manifolds and Kähler mapping torus, 
thereby highlighting the deep relationship between Hessian and Kähler geometries. 
We also study topological constraints, including parity properties of Betti numbers. 
Finally, we obtain a complete classification of compact, orientable $3$-dimensional Hessian manifolds.

\medskip
The paper is organized as follows.  
Section~2 recalls the main definitions and preliminary notions used throughout the work.  
In Section~3, we study $3$-dimensional compact orientable Hessian manifolds through their relation with Kähler geometry, 
and derive several topological consequences, including parity results for Betti numbers.  
Finally, Section~4 is devoted to the classification of compact, orientable $3$-dimensional Hessian manifolds.

\section{Basic Notions}
\label{sec:locally_flat_kv_cohomology}
All manifolds in this paper are assumed to be connected and without boundary, unless otherwise stated.

\subsection{Affine manifolds, Hessian metrics, and the first Koszul form}

\subsubsection{Affine manifolds}
\begin{defn}
A \emph{locally flat manifold} is a pair $(M, \nabla)$, where $\nabla$ is a connection whose torsion and curvature tensors vanish identically:
\[
T^\nabla = 0, \qquad R^\nabla = 0.
\]
Such a connection $\nabla$ is called a \emph{locally flat connection}, and $(M, \nabla)$ is referred to as a locally flat manifold.
\end{defn}

\begin{defn}
\label{def:affine_structure}
An \emph{affinely flat structure} on an $m$-dimensional manifold $M$ is a complete atlas
\[
\mathcal{A} = \{ (U_i, \Phi_i) \},
\]
such that the transition maps $\Phi_i^{-1} \circ \Phi_j$ belong to the group $\Aff(\mathbb{R}^m)$ of affine transformations on $\mathbb{R}^m$.
\end{defn}

\begin{rem}
In the language of $G$-structures, Definition~\ref{def:affine_structure} is equivalent to the statement that the structure group of $M$ reduces to $\Aff(\mathbb{R}^m)$. Several long-standing conjectures pertain to affine structures, including the Chern conjecture, which asserts that any compact affine manifold has vanishing Euler characteristic.
\end{rem}

A manifold admits an affine structure if and only if it is locally flat. This can be seen by considering local charts and the associated coordinate vector fields. In general, it is not possible to define a global connection from these vector fields, but in the presence of an affine structure, this becomes possible because the derivatives of the transition maps are constant. The resulting connection then has vanishing torsion and curvature. Conversely, if a connection on $M$ has vanishing torsion and curvature, the structure group reduces to $\Aff(\mathbb{R}^m)$, thereby inducing an affine structure.

\begin{ex}
The torus $\mathbb{T}^{n} = \mathbb{R}^{n} / \mathbb{Z}^{n}$ has a natural affine structure 
for which the projection map 
\(\pi : \mathbb{R}^n \to \mathbb{T}^{n}\) is an affine local diffeomorphism.
\end{ex}

\begin{ex}[\emph{Hopf manifold}]
Let $\lambda > 1$ and consider the action of $\mathbb{Z}$ on $\mathbb{R}^n \setminus \{0\}$ given by
\[
k \cdot x := \lambda^k x.
\]
Since this action is free and proper, the quotient is a smooth manifold called the \emph{Hopf manifold}.
\end{ex}

It is also interesting to provide examples of manifolds that do not admit an affine structure.

\begin{ex}[\cite{milnor19753,carriere1993inexistence,dal1991champs}]
Brieskorn manifolds $M(p,q,r)$ such that
\[
\frac{1}{p}+\frac{1}{q}+\frac{1}{r} < 1
\quad \text{or} \quad
pqr \text{ is even}.
\]
\end{ex}

\subsubsection{Hessian metrics}
Hessian manifolds possess particularly rich geometric and topological structures and are closely related to Kähler geometry.  
In particular, the tangent bundle of any Hessian manifold carries a natural Kähler metric induced by the Hessian structure
\cite{shima2007geometry,dombrowski1962geometry,satoh2019almost}.  
Hessian geometry plays a central role in information geometry, special Kähler geometry, and complex analysis, 
with important applications in convex optimization, statistics, combinatorics, Souriau’s thermodynamics on Lie groups, and quantitative finance.

\begin{defn}[\cite{shima1997geometry, shima2007geometry}]
A Riemannian metric $g$ on a locally flat manifold $(M, \nabla)$ is called a \emph{Hessian metric} if, locally, there exists a smooth function $\varphi$ such that
\[
g = \nabla^2 \varphi, \quad \text{i.e.,} \quad g_{ij} = \frac{\partial^2 \varphi}{\partial x_i \partial x_j},
\]
where $(x_1, \dots, x_n)$ is a system of affine coordinates with respect to $\nabla$.  

A locally flat manifold $(M, \nabla)$ equipped with a Hessian metric $g$ is called a \emph{Hessian manifold} and is denoted by $(M, g, \nabla)$.
\end{defn}

For examples of Hessian manifolds, see \cite{shima2007geometry} (p.~17) and \cite{calin2014geometric, ay2017information}.

Let $D$ denote the Levi-Civita connection of $(M,g)$, and define
\[
\gamma = D - \nabla.
\]
Since both $\nabla$ and $D$ are torsion-free, we have
\[
\gamma_X Y = \gamma_Y X, \quad \text{for all vector fields } X,Y.
\]
Moreover, in affine coordinates, the components $\gamma^{i}_{jk}$ of $s$ coincide with the Christoffel symbols $\Gamma^{i}_{jk}$ of $D$.

\begin{prop}[\cite{shima1997geometry}]
Let $(M, \nabla)$ be a locally flat manifold and $g$ a Riemannian metric. The following are equivalent:
\begin{enumerate}
    \item $g$ is a Hessian metric;
    \item $(\nabla_X g)(Y,Z) = (\nabla_Y g)(X,Z)$ for all vector fields $X,Y,Z$;
    \item In affine coordinates $(x^i)$, $g_{ij}$ satisfies
    \[
    \frac{\partial g_{ij}}{\partial x^k} = \frac{\partial g_{kj}}{\partial x^i};
    \]
    \item $g(\gamma_X Y, Z) = g(Y, \gamma_X Z)$ for all vector fields $X,Y,Z$;
    \item $\gamma_{ijk} = \gamma_{jik}$.
\end{enumerate}
\end{prop}

\begin{rem}
For a Hessian manifold \((M, g, \nabla)\), one can define a new affine connection by  
\[
\nabla^{*} = 2D - \nabla,
\]
where \(D\) is the Levi-Civita connection of $g$.  
The connection \(\nabla^{*}\) is locally flat and satisfies the compatibility condition  
\[
X \, g(Y, Z) = g(\nabla^{*}_{X} Y, Z) + g(Y, \nabla_{X} Z), \qquad X, Y, Z \in \mathfrak{X}(M).
\]
Hence, the triple \((M, g, \nabla^{*})\) is also a Hessian manifold, and the quadruple \((M, g, \nabla, \nabla^{*})\) is referred to as a \emph{dually flat manifold} \cite{amari2016information, ay2002dually, shima2007geometry, calin2014geometric}.  

It was shown in \cite{ay2002dually} that any compact dually flat manifold necessarily has infinite fundamental group. Consequently, the sphere \(\mathbb{S}^{d}\), for $d > 1$, cannot admit a dually flat structure.
\end{rem}

\subsubsection{The first Koszul form}
On a orientable Hessian manifold $(M,g,\nabla)$, using the connection $\nabla$ and the volume element induced by $g$, 
Koszul introduced a closed $1$-form, which Shima~\cite{shima2007geometry} calls the \emph{first Koszul form}, associated with the Hessian structure $(\nabla, g)$.

\begin{defn}[\cite{koszul1961domaines, shima1997geometry}]
The first Koszul form $\alpha$ associated with $(\nabla, g)$ is the $1$-form defined by
\[
\nabla_X (\mathrm{Vol}_g) = \beta(X) \, \mathrm{Vol}_g, \quad X \in \mathfrak{X}(M),
\]
where $\mathrm{Vol}_g$ is the Riemannian volume form induced by $g$.
\end{defn}

Directly from the definition, 
\[
\beta(X) = \operatorname{Tr}(\gamma_X),
\]
and in affine coordinates,
\[
\beta_i = \frac{1}{2} \frac{\partial \log \det(g_{pq})}{\partial x_i} = \gamma^k_{ki}.
\]

\subsection{Hyperbolic affine manifolds}
Jacques Vey extended W.~Kaup’s notion of hyperbolicity, originally developed for Riemann surfaces, 
to the setting of locally flat manifolds. In the sense of Vey~\cite{vey1969notion} and Kobayashi
\cite{kobayashi1977intrinsic,kobayashi1982invariant,kobayashi1984projectively}, 
hyperbolicity may be regarded as the opposite of geodesic completeness. 
This notion has been extensively studied by many authors, including Koszul
\cite{koszul1968deformations,koszul1961domaines,koszul1965varietes}, 
Vey~\cite{vey1969notion}, Vinberg~\cite{vinberg1967theory}, 
Shima~\cite{shima2007geometry}, and Kobayashi
\cite{kobayashi1977intrinsic,kobayashi1984projectively}. 
Hyperbolicity is a fundamental concept in Koszul geometry and admits several equivalent formulations.

\begin{defn}
A locally flat manifold $(M, \nabla)$ is called hyperbolic (affine hyperbolic) if its universal covering 
$(\widetilde{M}, \widetilde{\nabla})$ 
is diffeomorphic to a convex domain in the Euclidean space which does not contain any straight line.
(Koszul~\cite{koszul1968deformations,koszul1965varietes}).
\end{defn}

Koszul~\cite{koszul1968deformations} showed that, for a compact locally flat manifold, 
hyperbolicity is equivalent to the existence of a closed $1$-form $\beta$ 
whose covariant derivative $\nabla \beta$ is positive definite, 
thus defining a Riemannian metric on $M$. 
Since $\beta$ is everywhere nonzero, Tischler’s theorem~\cite{tischler1970fibering} implies that $M$
is a mapping torus. 
In particular, the Euler characteristic of $M$ vanishes and its first Betti number is positive
\cite{koszul1965varietes}.

\begin{rem}
Hyperbolic affine manifolds $(M,\nabla)$ are closely related to Hessian manifolds.
If $\beta$ is a closed $1$-form, its covariant derivative $\nabla\beta$ is a symmetric $2$-tensor.
By Koszul’s characterization of hyperbolicity, $\nabla\beta$ is positive definite and hence defines
a Riemannian metric on $M$. 
Since closed forms are locally exact, one can locally write $\beta=d\phi$ for some smooth function $\phi$,
and in this case $\nabla\beta$ coincides with the Hessian $d^2\phi$.
\end{rem}

Motivated by Koszul’s notion of hyperbolicity, Shima introduced in Hessian geometry the concept of
a Hessian structure of Koszul type.

\begin{defn}[\cite{shima2007geometry}]
A Hessian structure $(g,\nabla)$ is said to be \emph{of Koszul type} if there exists a closed $1$-form
$\eta$ such that
\[
g=\nabla\eta .
\]
\end{defn}

By Koszul’s results~\cite{koszul1968deformations}, if a compact manifold $M$ admits a Hessian structure
$(g,\nabla)$ of Koszul type, then the affine manifold $(M,\nabla)$ is hyperbolic. 
Equivalently, $M$ is diffeomorphic to a quotient
\[
M \cong \Xi/\Lambda,
\]
where $\Xi\subset\mathbb{R}^n$ is an open convex domain containing no affine line and
$\Lambda$ is a discrete subgroup of the affine group $\Aff(n)$
(see Koszul~\cite{koszul1968deformations} and Vey~\cite{vey1969notion}).

\subsection{Kähler mapping torus}
We now introduce one of the key notions of this work, which will play a fundamental role in the main results.
\begin{defn}
Let $\rho \in \mathrm{Diff}(L)$ be a self-diffeomorphism of a closed, connected manifold $L$. Then
\[
L_{\rho} = \frac{L \times [0,1]}{(p,0) \sim (\rho(p),1)}
\]
is the \emph{mapping torus} of $\rho$. It is the total space of a smooth fiber bundle
\[
L \hookrightarrow L_{\rho} \xrightarrow{\;\pi\;} \mathbb{S}^1.
\]
\end{defn}

If $(L, \Omega)$ is a symplectic manifold and $\rho$ is a symplectomorphism satisfying
\[
\rho^{*}\Omega = \Omega,
\]
then $L_{\rho}$ is called a \emph{symplectic mapping torus}. Moreover, if $(L, \Omega)$ is a Kähler manifold with associated Hermitian structure $(J, h)$, and if $\rho$ is a Hermitian isometry satisfying
\[
\rho^{*} \circ J = J \circ \rho^{*}, \quad \text{and} \quad \rho^{*}h = h,
\]
(which implies $\rho^{*}\Omega = \Omega$), then $L_{\rho}$ is called a \emph{Kähler mapping torus}.

\section{Hessian Manifolds, Kähler Mapping Torus and Betti Numbers
}
\subsection{Hessian manifolds and K\"ahler mapping torus}\label{sub:Kaler mapp}
Hessian manifolds are known to be closely related to K\"ahler manifolds. 
Indeed, the tangent bundle of a Hessian manifold carries a natural K\"ahler metric induced by the Hessian metric 
(see \cite{shima2007geometry,dombrowski1962geometry}).

This subsection is devoted to the study of the relationship, in dimension three, between compact orientable Hessian manifolds and K\"ahler manifolds.

\begin{thm}
\label{thm:class}
Any compact, orientable, $3$-dimensional Hessian manifold $(M,g,\nabla)$ is either the Hantzsche-Wendt manifold or a K\"ahler mapping torus.
\end{thm}

\begin{proof}
Let $(M,g,\nabla)$ be a compact, orientable Hessian manifold of dimension $3$, and let $D$ denote the Levi--Civita connection of $g$.  
By a theorem of Shima and Yagi~\cite[Theorem~4.1]{shima1996geometry}, the first Koszul form $\beta$ satisfies
\[
D\beta = 0,
\]
so its norm $\|\beta\|$ is constant on $M$. Consequently, exactly one of the following two cases occurs: either $\beta$ vanishes identically, or it is everywhere nonzero.

\medskip
\noindent
\textbf{Case 1: $\beta=0$.}
In this case, by~\cite[Theorem~4.2]{shima1996geometry}, the Levi--Civita and affine connections coincide, $D=\nabla$, so $(M,g)$ is flat.  
By the classical classification of compact flat $3$--manifolds due to Bieberbach, Hantzsche, Wendt, and Wolf
\cite{hantzsche1935dreidimensionale,Wolf,bieberbach1911bewegungsgruppen,bieberbach1912bewegungsgruppen},
the first homology group $H_1(M,\mathbb Z)$ is isomorphic to one of
\[
\mathbb Z^3,\qquad
\mathbb Z\oplus\mathbb Z_2^2,\qquad
\mathbb Z\oplus\mathbb Z_3,\qquad
\mathbb Z\oplus\mathbb Z_2,\qquad
\mathbb Z,\qquad
\mathbb{Z}_4 \oplus \mathbb{Z}_4 .
\]
Among the six compact flat $3$--manifolds, the only one with $b_1(M)=0$ is the Hantzsche--Wendt manifold, characterized by
\[
H_1(M,\mathbb Z)\cong\mathbb{Z}_4 \oplus \mathbb{Z}_4 
\]
(see e.g.~\cite{conway2006hearing}).  
All other cases satisfy $b_1(M)>0$. By Marrero and Padr\'on--Fern\'andez~\cite[Theorem~3.2]{marrero1998new}, $M$ is then diffeomorphic to one of
\[
M_1(1),\quad M_2(1),\quad M_1'(1),\quad M_2'(1),\quad \text{or}\quad \mathbb T^3,
\]
in the notation of~\cite{marrero1998new}. Moreover, each of these manifolds admits a co--K\"ahler structure. Since they are compact, it follows from Li~\cite[Theorem~2]{li2008topology} that each of them is a K\"ahler mapping torus.

\medskip
\noindent
\textbf{Case 2: $\beta \neq 0$.}
In this case, $\beta$ is a nowhere vanishing closed $1$--form. By Tischler’s theorem \cite{tischler1970fibering}, we deduce that $M$ is diffeomorphic to a mapping torus. By \cite[Proposition~3]{li2008topology}, any $3$--dimensional mapping torus is symplectic mapping torus; that is, it can be expressed as
\[
S_{\rho} = S \times_{\rho} \mathbb{S}^{1},
\]
where $(S,\omega)$ is a symplectic $2$--manifold and $\rho: S \to S$ is a symplectomorphism satisfying $\rho^{*}\omega = \omega$. Moreover, by \cite[Theorem~1]{li2008topology}, every compact symplectic mapping torus is a cosymplectic manifold. Consequently, there exists a closed $2$--form $\omega$ such that $\omega \wedge \beta > 0$.

Now, in our specific case, we can construct a co--K\"ahler structure using the Hessian metric $g$ as follows. Let $Z$ be the vector field defined by
\[
\beta = i_Z g.
\]
Since $D\beta = 0$, it follows that $DZ = 0$, and therefore $g(Z,Z) = k > 0$ is constant. Define the normalized vector field and $1$--form
\[
\widetilde{Z} = \frac{Z}{\sqrt{k}},
\qquad
\widetilde{\beta} = \frac{\beta}{\sqrt{k}}.
\]
Then
\[
\widetilde{\beta}(\widetilde{Z}) = 1,
\qquad
D\widetilde{Z} = 0,
\qquad
L_{\widetilde{Z}} g = 0,
\qquad
g(\widetilde{Z},\widetilde{Z}) = 1.
\]

Let $\Omega$ be the $2$--form defined by
\[
\Omega = i_{\widetilde{Z}} \mathrm{Vol}_g,
\]
where $\mathrm{Vol}_g$ denotes the Riemannian volume form of $g$. Since $\widetilde{Z}$ is parallel, we have
\[
d\Omega
= d(i_{\widetilde{Z}} \mathrm{Vol}_g)
= L_{\widetilde{Z}} \mathrm{Vol}_g
= 0.
\]
Moreover,
\[
i_{\widetilde{Z}}(\widetilde{\beta} \wedge \mathrm{Vol}_g) = 0,
\]
which implies
\[
\mathrm{Vol}_g = \widetilde{\beta} \wedge \Omega.
\]
Hence $(\widetilde{\beta}, \Omega)$ defines a cosymplectic structure on $M$, with Reeb vector field $\widetilde{Z}$. Since $\widetilde{Z}$ is Killing, $(M,\widetilde{\beta}, \Omega)$ is a $K$--cosymplectic manifold by \cite{bazzoni2015k}. By \cite[Proposition~2.8]{bazzoni2015k}, $M$ admits an almost contact metric structure $(\widetilde{\beta}, \phi, \widetilde{Z}, h)$ satisfying
\[
L_{\widetilde{Z}} h = 0.
\]
By Goldberg’s integrability theorem \cite[Proposition~3]{goldberg1969integrability}, this structure is normal; hence $(M, \widetilde{\beta}, \widetilde{Z}, \phi, h)$ is co--K\"ahler. Finally, Li’s theorem \cite[Theorem~2]{li2008topology} implies that $M$ is a K\"ahler mapping torus.
\end{proof}

\begingroup
\renewcommand\thefootnote{}
\footnotetext{A cosymplectic structure on an oriented $3$--dimensional manifold is a pair $(\lambda,\omega)$ such that $d\omega=0$, $d\lambda=0$, and $\lambda \wedge \omega > 0$~\cite{li2008topology,libermann1959automorphismes}.}

\footnotetext{A $K$--cosymplectic structure on an oriented $3$--dimensional manifold is a pair $(\lambda,\omega)$ such that the Reeb vector field is Killing with respect to some Riemannian metric on $M$~\cite{bazzoni2015k}.}

\footnotetext{A co--K\"ahler structure on an oriented $3$--dimensional manifold consists of a quadruple $(\lambda,\xi,\phi,g)$ whose fundamental $2$-form $\omega$ associated with the structure is defined by
$$
\omega(X,Y) = g(X,\phi Y).
$$
Such that $d\lambda=0$, $d\omega=0$, and the tensor $[\phi,\phi]$ vanishes, where $[\phi,\phi]$ is the Nijenhuis tensor of $\phi$~\cite{blair2010riemannian,li2008topology}.}
\endgroup

\begin{com}
Note that in the case where $\beta$ does not vanish, the dual vector field $Z$ is always the Reeb vector field of a co-K\"ahler structure. In terms of vector bundle isomorphisms, one has
\[
Z = \chi_{\beta,\Omega}^{-1}(\beta),
\]
where
\[
\chi_{\beta,\Omega} : TM \longrightarrow T^*M, \qquad 
v_q \longmapsto \iota_{v_q}\Omega + \beta(v_q)\,\beta.
\]

Moreover, it is well known from~\cite{de1993cosymplectic, cantrijn1992gradient} that there exist canonical coordinates $(z,p,q)$ adapted to the Reeb flow, in which the vector field $Z$ and the first Koszul form $\beta$ take the simple forms
\[
Z = \frac{\partial}{\partial z},
\qquad
\beta = dz.
\]
\end{com}

\begin{rem}
    Note that, as pointed out by Shima~\cite{shima1997geometry} (p.~282, Example), for non-compact Hessian manifolds the condition $D \beta = 0$ is not always satisfied.  
\end{rem}

We can now state the following result.

\begin{cor}
\label{thm:k1}
Any compact, orientable, $3$-dimensional Hessian manifold is finitely covered by a product $S \times \mathbb{S}^1$, where $S$ is a compact K\"ahler surface.
\end{cor}

\begin{proof}
Let $(M, g, \nabla)$ be a compact, orientable, $3$-dimensional Hessian manifold. By Theorem~\ref{thm:class}, $M$ is either a Hantzsche--Wendt manifold (a flat Riemannian manifold) or a K\"ahler mapping torus. In the first case, Bieberbach's theorem~\cite{bieberbach1911bewegungsgruppen} implies that $M$ is finitely covered by a torus $\mathbb{T}^3$, which can be written as $\mathbb{T}^2 \times \mathbb{S}^1$. In the second case, by~\cite[Theorem~3.4]{bazzoni2014structure}, $M$ is finitely covered by a product $S \times \mathbb{S}^1$, where $S$ is a compact K\"ahler surface.
\end{proof}

\begin{cor}\label{th:Hant}
Let $(M,g,\nabla)$ be a compact, orientable, $3$-dimensional Hessian manifold.  
If the first Betti number vanishes, i.e., $b_1(M)=0$, then $M$ is diffeomorphic to the Hantzsche-Wendt manifold, and
\[
H_1(M, \mathbb{Z}) \cong \mathbb{Z}_4 \oplus \mathbb{Z}_4 .
\]
\end{cor}

\begin{proof}
By Theorem~\ref{thm:class}, $M$ is either a flat Hessian manifold ($\beta=0$) or a K\"ahler mapping torus ($\beta \neq 0$). Assume that $b_1(M)=0$. If $\beta$ were nowhere vanishing, then $\beta=df$ for some smooth function $f: M \to \mathbb{R}$. Since $M$ is compact, $f$ attains a maximum at some point $p \in M$, hence $df(p)=0$, contradicting the assumption that $\beta$ is nowhere vanishing. Therefore, $\beta$ must vanish identically, which implies $\nabla = D$ by \cite{shima2007geometry}. Hence, $(M,D)$ is a flat Riemannian manifold. Recall that a compact, orientable, flat $3$-manifold belongs to one of the six classes $\mathcal{S}_1,\dots,\mathcal{S}_6$ in Wolf's classification~\cite{Wolf}(page 122, Chap. 3). Moreover, the first integral homology group $H_1(M,\mathbb{Z})$ is isomorphic to one of
\[
\mathbb{Z}^3,\quad 
\mathbb{Z}\oplus\mathbb{Z}_2^2,\quad 
\mathbb{Z}\oplus\mathbb{Z}_3,\quad 
\mathbb{Z}\oplus\mathbb{Z}_2,\quad 
\mathbb{Z},\quad 
\mathbb{Z}_4\oplus\mathbb{Z}_4 .
\]
The first five cases all have positive first Betti number. Since $b_1(M)=0$, $M$ must belong to the sixth class $\mathcal{S}_6$, which is the Hantzsche--Wendt manifold. Therefore,
\[
H_1(M,\mathbb{Z}) \cong \mathbb{Z}_4 \oplus \mathbb{Z}_4 .
\]
\end{proof}

\begin{cor}
The product of two compact, orientable, $3$-dimensional Hessian manifolds of Koszul type (hyperbolic) admits a K\"ahler structure. 
\end{cor}
\begin{proof}
By \cite{koszul1968deformations}, the first Betti number of a compact affine hyperbolic manifold $(M,\nabla)$ is positive, $b_1(M)>0$. By Theorem~\ref{thm:class} and Corollary~\ref{th:Hant}, $M$ is a co--K\"ahler manifold with co--K\"ahler structure $(\widetilde{\beta_{1}},\phi_{1},\widetilde{Z_{1}},h_{1})$. Consider another compact affine hyperbolic manifold $(N,\widetilde{\nabla})$ with co--K\"ahler structure $(\widetilde{\beta_{2}},\phi_{2},\widetilde{Z_{2}},h_{2})$. Following \cite{watson1983new, morimoto1963normal}, for each $a\in\mathbb{R}$ and $b\neq 0$, the Cartesian product $M \times N$ carries a K\"ahler structure $(J_{a,b}, G_{a,b})$ defined by
\begin{align}
J_{a,b}(X_1, X_2) &= 
\Bigl(
\phi_1 X_1 - \Bigl(\frac{a}{b} \, \widetilde{\beta_{1}}(X_1) + \frac{a^2 + b^2}{b} \, \widetilde{\beta_{2}}(X_2)\Bigr)\, \widetilde{Z_{1}}, \nonumber\\
& \phi_2 X_2 + \Bigl(\frac{1}{b} \, \widetilde{\beta_{1}}(X_1) + \frac{a}{b} \, \widetilde{\beta_{2}}(X_2)\Bigr)\, \widetilde{Z_{2}}
\Bigr),\nonumber\\ 	 
G_{a,b}((X_1, X_2), (Y_1, Y_2)) &= 
h_1(X_1, Y_1) + a \, \widetilde{\beta_{1}}(X_1)\widetilde{\beta_{2}}(Y_2) + a \, \widetilde{\beta_{1}}(Y_1)\widetilde{\beta_{2}}(X_2) \nonumber\\
& + (a^2 + b^2 + 1) \, \widetilde{\beta_{2}}(X_2)\widetilde{\beta_{2}}(Y_2) + h_2(X_2, Y_2).\nonumber
\end{align}  
\end{proof}

\begin{rem}
We can perturb the first Koszul form $\beta$ by adding a small closed $1$-form to obtain a new closed $1$-form $\beta_0$ such that:
\begin{itemize}
    \item $[\beta_0] \in H^1(M;\mathbb{R})$ is a real multiple of an integral cohomology class in $H^1(M;\mathbb{Z})$;
    \item $\beta_0 \wedge \Omega > 0$ everywhere on $M$.
\end{itemize}
Since $[\beta_0]$ is proportional to an integral class, there exists a fibration
\[
\pi \colon M \to \mathbb{S}^1
\]
such that $\beta_0$ is a constant multiple of $\pi^*(\mathrm{vol}_{\mathbb{S}^1})$, where $\mathrm{vol}_{\mathbb{S}^1} = d\theta$ is the standard volume form on the circle. The condition $\beta_0 \wedge \Omega > 0$ ensures that $\Omega$ restricts to a symplectic form on each fiber of $\pi$.
\end{rem}

\subsection{Betti numbers of $3$--dimensional Hessian manifolds}\label{section:Betti}

First, recall that, as already pointed out in the proof of Theorem~\ref{thm:class} by a direct computation, for an $n$--dimensional orientable Hessian manifold whose first Koszul form does not vanish, we have $DZ = 0$, where $Z$ is the vector field dual to the first Koszul form $\beta$. This fact has important consequences for the Betti numbers of any compact orientable Hessian manifold of arbitrary dimension $n$.  

From Chern~\cite{chern1966geometry} (Example~3), one can deduce that the first two Betti numbers of a compact orientable Hessian manifold whose first Koszul form does not vanish necessarily satisfy
\[
b_1(M) \ge 1, \qquad b_2(M) \ge b_1(M) - 1,
\]
a condition now known as Chern’s criterion. By the work of Karp~\cite{karp1977parallel}, additional necessary conditions on the Betti numbers extending Chern’s result were derived, namely
\[
b_1(M) \ge 1, \qquad 
b_{l+1}(M) \ge b_l(M) - b_{l-1}(M), \quad \text{for all } 1 \le l \le n-1.
\]

In this subsection, we study some parity properties of the Betti numbers of $3$--dimensional orientable Hessian manifolds, with particular emphasis on those of Koszul type.

\begin{thm}
\label{thm:bett1}
The first Betti number of any compact, orientable, $3$--dimensional Hessian manifold $(M,g,\nabla)$ is either zero or odd. In the case of a Hessian manifold of Koszul type (hyperbolic), all Betti numbers are odd.
\end{thm}

\begin{proof}
By Theorem~\ref{thm:class}, any compact, orientable, $3$--dimensional Hessian manifold is either the Hantzsche--Wendt manifold or a K\"ahler mapping torus.
In the Hantzsche--Wendt case, one has $b_1(M)=0$. In the K\"ahler mapping torus case, by~\cite[Theorem~2]{li2008topology}, $M$ is co--K\"ahler. Furthermore, by~\cite[Theorem~11]{Chinea1993TopologyOC}, the first Betti number $b_1(M)$ of a compact co--K\"ahler manifold is odd. Hence, in all cases, $b_1(M)$ is either zero or odd.

Now let $(M,g,\nabla)$ be a compact Hessian $3$--manifold of Koszul type. By~\cite{koszul1968deformations}, the affine manifold $(M,\nabla)$ is hyperbolic, and therefore $b_1(M)>0$. By the first part of the theorem, $b_1(M)$ must be odd, so $b_1(M)=2p+1$ for some $p\ge 0$. By Poincar\'e duality, we have $b_2(M)=b_1(M)$ and $b_0(M)=b_3(M)=1$, which proves the result.
\end{proof}

\begin{com}
The first Betti number of the Heisenberg manifold 
\[
\mathrm{Heis}_{\mathbb{Z}} \backslash \mathrm{Heis}_{\mathbb{R}}
\]
is equal to $2$. Therefore, by Theorem~\ref{thm:bett1}, the Heisenberg manifold cannot admit a Hessian structure.

The connected sum of two real projective 3-spaces $\mathbb{RP}^3 \# \mathbb{RP}^3$ has first Betti number $b_1(M)=0$, and the same holds for the lens spaces $L_{p,q}$ and for the sphere $\mathbb{S}^{3}$. However, none of these manifolds is diffeomorphic to the Hantzsche--Wendt manifold. Hence, by Corollary~\ref{th:Hant}, they cannot admit a Hessian structure.
\end{com}

\subsection{Topology of $3$--dimensional Hessian manifolds with non-vanishing first Koszul form}

In this subsection, we focus on the case where the first Koszul form $\beta$ is non-vanishing, and we study the dynamical system generated by the vector field $Z$ defined by $\beta = i_{Z}g$. In the compact case, when the first Koszul form does not vanish, the vector field $Z$ generates a global one-parameter group $\mathrm{Exp}(tZ)$ of diffeomorphisms of $M$. For simplicity, we refer to $\mathrm{Exp}(tZ)$ as the flow of $Z$. In what follows, we study the dynamical system induced by this flow.

\begin{thm}
Let $(M,g,\nabla)$ be a compact, orientable, $3$-dimensional Hessian manifold whose first Koszul form $\beta$ is non-vanishing. If the dual vector field $Z$ of $\beta$ is aperiodic, then $M$ is diffeomorphic to a torus bundle $\mathbb{T}^{2}_{U}$ with $\mathrm{tr}(U)=2$. In the case where the orbits of $Z$ are dense, $M$ is diffeomorphic to the $3$-torus $\mathbb{T}^{3}$.
\end{thm}

\begin{proof}
By Theorem~\ref{thm:class}, $M$ is a K\"ahler mapping torus and therefore admits a co--K\"ahler structure. This naturally induces a cosymplectic structure $(\Omega,\beta)$, whose Reeb vector field is $Z$. Note that a cosymplectic manifold is a special case of a stable Hamiltonian structure with $f=0$~\cite{hutchings2009weinstein}.  

By assumption, $Z$ is aperiodic. Then, by~\cite[Theorem~1.1]{hutchings2009weinstein}, $M$ is a $\mathbb{T}^2$--bundle over $\mathbb{S}^{1}$. Consequently, $Z$ is orbit--equivalent to the suspension of a symplectomorphism of the torus
\[
\phi : \mathbb{T}^2 \longrightarrow \mathbb{T}^2,
\]
whose mapping class corresponds to a matrix $U \in \mathrm{SL}(2,\mathbb{R})$. The Lefschetz number of $\phi$ is
\[
\Lambda_\phi = \sum_{i=0}^{2} (-1)^i \, \mathrm{tr}\bigl(\phi_* H_i(\mathbb{T}^2,\mathbb{Q})\bigr)
= 2 - \mathrm{tr}(U).
\]
By hypothesis, $\phi$ admits no periodic points, which implies $\Lambda_\phi = 0$, and hence $\mathrm{tr}(U)=2$.  

In the case where the orbits of $Z$ are dense, by~\cite[Th\'eor\`eme~1]{carriere1984flots}, we deduce that $M$ is diffeomorphic to the torus $\mathbb{T}^{3}$.
\end{proof}

\begin{rem}
If a $3$--dimensional Hessian manifold $(M,g,\nabla)$ has first Betti number $b_1(M)=1$, then, by Subsection~\ref{sub:Kaler mapp}, the dual vector field $Z$ is the Reeb vector field of a co--K\"ahler structure, and hence of a $K$--cosymplectic structure. According to~\cite[Theorem~7.6]{bazzoni2015k}, the $\mathbb{S}^1$--action generated by $Z$ is automatically Hamiltonian. In particular, $Z$ admits at least two closed Reeb orbits.

Moreover, the number of closed Reeb orbits cannot be finite. Indeed, as noted in the discussion following~\cite[Corollary~8.7]{bazzoni2015k}, if the number of closed Reeb orbits $N$ is finite, then
\[
N = \dim H_B^*(M,\mathcal{F}_Z),
\]
where $H_B^*(M,\mathcal{F}_Z)$ denotes the basic cohomology of $M$ with respect to the foliation $\mathcal{F}_Z$ induced by $Z$. By~\cite[Corollaries~8.7 and~8.9]{bazzoni2015k}, one has
\[
N = 1 + b_1(M) = 2.
\]
By~\cite[Th\'eor\`eme~1]{carriere1984flots}, it follows that $M$ is either diffeomorphic to a lens space $L_{p,q}$ or to $\mathbb{S}^2 \times \mathbb{S}^1$. As shown above, $M$ cannot be diffeomorphic to a lens space. Moreover, by~\cite{shima1981hessian}, the manifold $\mathbb{S}^2 \times \mathbb{S}^1$ has universal cover $\mathbb{S}^2 \times \mathbb{R}$ and therefore cannot admit a Hessian structure. Hence, in this situation, all orbits of $Z$ must be closed.
\end{rem}

\subsection*{Regular case}

We now consider the case where the foliation $\mathcal{F}_{Z}$ generated by $Z$ is regular in the sense of Palais~\cite{palais1957global}. For simplicity, we shall refer to $Z$ itself as regular.

\begin{thm}
Let $(M,g,\nabla)$ be a compact, orientable, $3$--dimensional Hessian manifold whose first Koszul form $\beta$ is non-vanishing. If the dual vector field $Z$ of $\beta$ is regular, then $M$ is diffeomorphic to the product
\[
M \simeq B \times \mathbb{S}^{1},
\]
where $B$ is a compact K\"ahler manifold (necessarily a Riemann surface).
\end{thm}

\begin{proof}
Since $Z$ is regular, each leaf of the foliation $\mathcal{F}_{Z}$ is diffeomorphic to the circle $\mathbb{S}^{1}$, and the orbit space $B = M/\mathcal{F}_{Z}$ is a smooth manifold. Therefore, $M$ is the total space of a principal $\mathbb{S}^{1}$--bundle
\[
\pi : M \longrightarrow B,
\]
and the first Koszul form $\beta$ defines a connection $1$--form on this bundle. Since $\beta$ is closed, its curvature form vanishes, and hence the bundle is flat. Moreover, since $Z$ is the Reeb vector field of a co--K\"ahler structure $(\beta,Z,\phi,g)$ with fundamental $2$--form $\Phi$, it follows from~\cite{blair2010riemannian,li2008topology} that
\[
\mathcal{L}_{Z}\phi = 0, \qquad 
\mathcal{L}_{Z}g = 0, \qquad 
\mathcal{L}_{Z}\Phi = 0.
\]
Therefore, there exists a K\"ahler structure $(\widetilde{\phi},\widetilde{g},\widetilde{\Phi})$ on $B$ such that
\[
\pi^{*}\widetilde{\phi} = \phi, \qquad 
\pi^{*}\widetilde{g} = g, \qquad 
\pi^{*}\widetilde{\Phi} = \Phi.
\]

The Euler class $e_M \in H^{2}_{B}(\mathcal{F}_{Z})$ of the principal $\mathbb{S}^{1}$--bundle $\pi : M \to B$ is represented (up to a constant factor) by the curvature form of $\beta$, and hence vanishes. By~\cite{greub1972connections} or~\cite[Exercise~10.2.7]{martelli2016introduction}, a principal $\mathbb{S}^{1}$--bundle with vanishing Euler class admits a global section and is therefore trivial. Consequently,
\[
M \simeq B \times \mathbb{S}^{1},
\]
where $B$ is a compact K\"ahler manifold of real dimension $2$.
\end{proof}

\section{Complete classification of $3$-dimensional compact orientable Hessian manifolds}

This subsection provides a complete classification of $3$-dimensional orientable Hessian manifolds in terms of the eight Thurston geometries. 

\begin{thm}
\label{thm:sei}
Any compact, orientable, $3$-dimensional Hessian manifold $(M, g, \nabla)$ is either a Hantzsche--Wendt manifold; the $3$-torus $\mathbb{T}^{3}$; one of the torus bundles $\mathbb{T}^{2}_{U}$ with
\[
U \in \left\{ 
\begin{pmatrix}0 & 1 \\ -1 & 0\end{pmatrix},\ 
\begin{pmatrix}-1 & 0 \\ 0 & -1\end{pmatrix},\ 
\begin{pmatrix}0 & -1 \\ 1 & 1\end{pmatrix},\ 
\begin{pmatrix}-1 & -1 \\ 1 & 0\end{pmatrix} 
\right\};
\] 
or a quotient $(\mathbb{H}^{2} \times \mathbb{R}) / \Pi$, where $\Pi \subset \mathrm{Isom}(\mathbb{H}^{2} \times \mathbb{R})$ is a discrete subgroup acting freely. 
\end{thm}

\begin{proof}
Let $(M, g, \nabla)$ be a compact, orientable Hessian $3$-manifold. By Theorem~\ref{thm:class}, $M$ is either the Hantzsche--Wendt manifold when $b_1(M) = 0$, or a Kähler mapping torus when $b_1(M) > 0$. In the latter case, by~\cite{li2008topology}, $M \simeq \Sigma_\varphi$, where $(\Sigma, J, h)$ is a Kähler manifold. Let $\mathrm{Isom}(\Sigma,h)$ denote the group of Hermitian isometries of $\Sigma$, and $\mathrm{Isom}_0(\Sigma,h)$ its identity component. By~\cite[Theorem~6.6]{bazzoni2014structure}, the diffeomorphism $\varphi$ is either isotopic to the identity or of finite order in $\mathrm{Isom}(\Sigma,h)/\mathrm{Isom}_0(\Sigma,h)$; in particular, $\varphi$ is periodic. By~\cite[Theorem~5.4]{scott1983geometries}, $M$ is a Seifert fibered space over a $2$-dimensional orbifold $B$ with Euler number zero. The geometries $\mathbb{S}^3$, $\mathrm{Nil}$, and $\widetilde{\mathrm{SL}_2(\mathbb{R})}$ are excluded.

\medskip
\noindent \textbf{Case 1: $\chi(B) > 0$.}  
In this case, $M$ has $\mathbb{S}^2 \times \mathbb{R}$ geometry. By~\cite[Proposition~10.3.36]{martelli2016introduction}, $M$ is diffeomorphic to one of $\mathbb{S}^2 \times \mathbb{S}^1$, the orientable $\mathbb{S}^1$-bundle over $\mathbb{RP}^2$, or $\bigl(S^2; (p,q), (p,-q)\bigr)$. All manifolds of the last type are diffeomorphic to $\mathbb{S}^2 \times \mathbb{S}^1$. Since $M$ carries a Hessian structure, $\mathbb{S}^2 \times \mathbb{S}^1$ is excluded~\cite{shima1981hessian}. The orientable $\mathbb{S}^1$-bundle over $\mathbb{RP}^2$ is diffeomorphic to $\mathbb{RP}^3 \# \mathbb{RP}^3$, which has $b_1(M) = 0$ and cannot admit a Hessian structure~\cite[Theorem~1]{smillie1981obstruction} (note that one may also use Subsection~\ref{section:Betti}).

\medskip
\noindent \textbf{Case 2: $\chi(B) = 0$.}  
In this case, $M$ has Euclidean geometry. By~\cite[Theorem~12.3.1]{martelli2016introduction}, $M$ admits a flat metric. From the classification of closed flat $3$-manifolds~\cite{hantzsche1935dreidimensionale,Wolf,bieberbach1911bewegungsgruppen,bieberbach1912bewegungsgruppen}, the first homology group $H_1(M,\mathbb{Z})$ is one of $\mathbb{Z}^3$, $\mathbb{Z} \oplus \mathbb{Z}_2^2$, $\mathbb{Z} \oplus \mathbb{Z}_3$, $\mathbb{Z} \oplus \mathbb{Z}_2$, $\mathbb{Z}$, or $\mathbb{Z}_4 \oplus \mathbb{Z}_4$. The Hantzsche--Wendt manifold corresponds to $H_1(M,\mathbb{Z}) \cong \mathbb{Z}_4 \oplus \mathbb{Z}_4$ with $b_1(M)=0$. All other manifolds have $b_1(M) > 0$ and are diffeomorphic to $M_1(1)$, $M_2(1)$, $M_1'(1)$, $M_2'(1)$, or $\mathbb{T}^3$~\cite[Theorem~3.2]{marrero1998new}. These correspond to the torus bundles:
\[
M_1(1) \simeq \mathbb{T}^2_U,\ U = \begin{pmatrix}0 & 1 \\ -1 & 0\end{pmatrix},\quad
M_2(1) \simeq \mathbb{T}^2_U,\ U = \begin{pmatrix}-1 & 0 \\ 0 & -1\end{pmatrix},
\]
\[
M_1'(1) \simeq \mathbb{T}^2_U,\ U = \begin{pmatrix}0 & -1 \\ 1 & 1\end{pmatrix},\quad
M_2'(1) \simeq \mathbb{T}^2_U,\ U = \begin{pmatrix}-1 & -1 \\ 1 & 0\end{pmatrix}.
\]

\medskip
\noindent \textbf{Case 3: $\chi(B) < 0$.}  
In this case, $M$ has $\mathbb{H}^2 \times \mathbb{R}$ geometry, and hence $M \simeq (\mathbb{H}^2 \times \mathbb{R}) / \Pi$, where $\Pi \subset \mathrm{Isom}(\mathbb{H}^2 \times \mathbb{R})$ is a discrete subgroup acting freely~\cite{scott1983geometries}.
\end{proof}

\subsection*{Radiant affine manifolds and Hessian geometry}

We conclude this article by relating our classification results to the theory of radiant affine manifolds.

\begin{defn}\cite{choi2001decomposition,goldman2022geometric,shima1997geometry}
Let $(M,\nabla)$ be an affine manifold. A vector field $H$ on $M$ is called an
\emph{Euler vector field} of $(M,\nabla)$ if
\[
\nabla_X H = X
\]
for every vector field $X$ on $M$. An affine manifold is said to be
\emph{radiant} if it admits an Euler vector field.
\end{defn}

\begin{com}
The classification of radiant affine $3$-manifolds was established in~\cite{choi2001decomposition}. 
Building on the work of Barbot~\cite{barbot2000varietes}, Choi proved that every radiant affine $3$-manifold is a radiant suspension. 
In particular, any such manifold is either a Seifert $3$-manifold covered by a product $S \times \mathbb{S}^1$, where $S$ is a closed surface, a Heisenberg nilmanifold, or a hyperbolic torus bundle.

From Theorem~\ref{thm:sei}, if $M$ admits a Hessian structure, then $M$ must be a Seifert manifold. 
However, a hyperbolic torus bundle cannot be a Seifert manifold~\cite[Proposition~11.4.14]{martelli2016introduction}. 
Moreover, if $M$ is a Heisenberg nilmanifold, then its first Betti number is equal to $2$, which contradicts Proposition~\ref{thm:bett1}. 
Therefore, combining with Theorem~\ref{thm:sei}, we conclude that the only radiant affine $3$-manifolds admitting Hessian metrics are those for which $S$ is either the $2$-torus $\mathbb{T}^2$ or a closed hyperbolic surface $\Sigma_g$ with $g \geq 2$.

\end{com}

\medskip
\noindent
In conclusion, Hessian geometry imposes strong topological and geometric rigidity on radiant affine $3$-manifolds. 
The only examples arise from Seifert manifolds with trivial Euler number, whose
underlying orbifold has either vanishing or negative Euler characteristic, and the resulting geometries reduce to the product geometries
$\mathbb{E}^3$ and $\mathbb{H}^2 \times \mathbb{R}$. 
This provides a complete classification of radiant affine $3$-manifolds admitting Hessian metrics.

\bibliographystyle{splncs04}
\bibliography{main.bib}
\end{document}